\documentclass[reqno,a4paper,11pt]{amsart}

\usepackage[pdfpagelabels]{hyperref}
\usepackage{amssymb}

\newtheorem{theorem}{Theorem}
\newtheorem{lemma}[theorem]{Lemma}
\newtheorem{corollary}[theorem]{Corollary}
\newtheorem{oldtheorem}{Theorem}
\newtheorem{oldlemma}[oldtheorem]{Lemma}

\theoremstyle{definition}

\theoremstyle{remark}
\newtheorem{remark}{Remark}
\newtheorem{example}{Example}

\newcommand{\set}[1]{\left\{#1\right\}}
\newcommand{\abs}[1]{\lvert#1\rvert}
\newcommand{\nm}[1]{\lVert#1\rVert}
\newcommand{\bnm}[1]{\big\lVert#1\big\rVert}

\newcommand{\D}{\mathbb{D}}

\newcommand{\N}{\mathbb{N}}
\newcommand{\Z}{\mathbb{Z}}

\newcommand{\C}{\mathbb{C}}

\renewcommand{\phi}{\varphi}

\newcommand{\BMOA}{\rm BMOA}

\begin{document}

\title[On non-normal solutions of linear differential equations]{On non-normal solutions of linear\\ differential equations}

\author{Janne Gr\"ohn}
\thanks{The author is supported in part by the Academy of Finland, projects \#258125 and \#286877.}
\address{Department of Physics and Mathematics, 
University of Eastern Finland,\newline \indent P.O. Box 111, FI-80101 Joensuu, Finland}
\email{janne.grohn@uef.fi}

\date{\today}


\begin{abstract} 
Normality arguments are applied to study the oscillation of solutions of
$f''+Af=0$, where the coefficient $A$ is analytic
in the unit disc $\D$ and  $\sup_{z\in\D} (1-|z|^2)^2|A(z)| < \infty$.
It is shown that such differential equation may admit a non-normal solution
having prescribed uniformly separated zeros.
\end{abstract}


\maketitle


\section{Introduction} \label{sec:introduction}

The purpose of this paper is to consider the oscillation of solutions of 
\begin{equation} \label{eq:de2}
f'' + A f = 0,
\end{equation}
where the coefficient $A$ is analytic in the unit disc $\D$. Due to an extensive existing literature on
the subject, zero-sequences of individual solutions of \eqref{eq:de2}
can be described in various ways. However, it is curious how little is known of 
the geometric zero distribution of the product of two linearly independent solutions of \eqref{eq:de2}.

Let $f_1$ and $f_2$ be linearly independent solutions of \eqref{eq:de2}. By \eqref{eq:de2},
the Wronskian determinant  $W(f_1,f_2) = f_1 f_2' - f_1'f_2$
is a non-zero complex constant. 
We deduce:
\begin{itemize}
\item[\rm (i)]
Zeros of $f_1$ and $f_2$ are simple.
\item[\rm (ii)]
Zeros of $f_1f_2$ are simple.
\end{itemize}
Our plan is to elaborate these conclusions. Concerning (i), we focus on the separation
between zeros of a non-trivial solution and zeros of its derivative. Discussion of (ii) leads us to the
concept of normality (in the sense of Lehto and Virtanen). As a~main result, we construct a coefficient $A\in H^\infty_2$ such that
\eqref{eq:de2} admits a non-normal solution having prescribed uniformly separated zeros. 
Here $H^\infty_2$ is the space of those analytic functions $g$
in $\D$ for which
\begin{equation*}
  \nm{g}_{H^\infty_2} = \sup_{z\in\D} \, (1-|z|^2)^2 |g(z)| <\infty.
\end{equation*}
Finally, we consider the normality of $w=f_1/f_2$, and obtain an estimate for the 
separation of zeros of $f_1f_2$ in the case $A\in H^\infty_2$.


\section{Separation of critical points}

Our intention is to discuss the extent to which the separation properties of zeros of
solutions of \eqref{eq:de2}
hold true for their critical points. The critical points of an~analytic function $f$
are the zeros of the derivative $f'$. In this paper, separation always refers to
the separation with respect to the hyperbolic metric.

Let $f$ be a non-trivial ($f\not\equiv 0)$ solution of \eqref{eq:de2} in the unit disc $\D$. We may ask
the following questions, and consider their relation to the growth of the coefficient.
\begin{enumerate}
\item[(Q1)] Are the zeros of $f$ separated?
\item[(Q2)] Are the critical points of $f$ separated?
\item[(Q3)] Are the zeros of $f$ separated from the critical points of $f$?
\end{enumerate}

We proceed to consider (Q1)-(Q3) under certain restrictions for the growth of the coefficient
$A$. If $\psi:[0,1)\to (0,1)$ is a non-increasing function such that
\begin{equation} \label{eq:smoothness}
K = \sup_{0\leq r < 1} \, \frac{\psi(r)}{\psi\left( \frac{r+\psi(r)}{1+r\psi(r)} \right)} < \infty,
\end{equation}
and $A$ is an analytic function satisfying
\begin{equation} \label{eq:grest}
\sup_{z\in\D} \, \abs{A(z)} \big( \psi(\abs{z}) (1-\abs{z}^2) \big)^2 = M < \infty,
\end{equation}
then (Q1) admits a complete answer according to 
\cite[Theorem~11]{CGHR:2013}. In particular, these assumptions
imply that any distinct  zeros $\zeta_1,\zeta_2\in\D$ of any non-trivial solution $f$ of \eqref{eq:de2}
are separated in the hyperbolic metric by
\begin{equation*}
\varrho_h(\zeta_1,\zeta_2) 
      \geq \log \frac{1+\psi\big( | t_h(\zeta_1,\zeta_2)|\big)/\max\{ K\sqrt{M}, 1\}}
      {1-\psi\big(| t_h(\zeta_1,\zeta_2)| \big)/\max\{ K\sqrt{M}, 1\}},
\end{equation*}
and vice versa. Here $t_h(\zeta_1,\zeta_2)$ is the hyperbolic mid-point of $\zeta_1$ and $\zeta_2$;
\begin{equation*}
\varrho_h(\zeta_1,\zeta_2)= \frac{1}{2}\log\frac{1+\varrho_p(\zeta_1,\zeta_2)}{1-\varrho_p(\zeta_1,\zeta_2)}
\quad \text{and} \quad
\varrho_p(\zeta_1,\zeta_2)= \abs{\varphi_{\zeta_1}(\zeta_2)}
\end{equation*}
are the hyperbolic and the pseudo-hyperbolic distances between $\zeta_1$ and $\zeta_2$;
and $\varphi_a(z)=(a-z)/(1-\overline{a}z)$, $a\in\D$, 
is an automorphism of $\D$, which coincides with its own inverse.
We refer to \cite[Section~2.3]{CGHR:2013} for a~detailed study of the smoothness condition~\eqref{eq:smoothness}.
The separation result above is an extension of the classical findings \cite[Theorem~3-4]{S:1955} by B.~Schwarz:
if $A\in H^\infty_2$  then the hyperbolic distance between any distinct zeros of any non-trivial 
solution of \eqref{eq:de2} is uniformly bounded away from zero by a~constant depending on $\nm{A}_{H^\infty_2}$, 
and vice versa.

The question (Q2) admits an immediate negative answer, which is independent of the growth of the coefficient. 
For example, \eqref{eq:de2}
for $A(z)=-6z/(z^3+2)$ admits a solution $f(z)=z^3 +2$, whose derivative has
a~two-fold zero at the origin. 
Even more is true. The following example proves that, if $A\in H^\infty_2$ then 
zeros of the derivative of a solution of \eqref{eq:de2} can have arbitrarily high multiplicity. 
Moreover, there is no lower bound even for the separation of \emph{distinct} critical points.


\begin{example}
Let $\{\zeta_n\}_{n=1}^\infty \subset \D$ be a Blaschke-sequence, i.e. $\sum_{n=1}^\infty(1-|\zeta_n|)<\infty$, 
and consider the Blaschke product
\begin{equation} \label{eq:Blaschke}
B(z) = \prod_{n=1}^\infty \frac{|\zeta_n|}{\zeta_n} \, \frac{\zeta_n-z}{1-\overline{\zeta}_n z}, 
\quad z\in\D.
\end{equation}
Here we take the convention that $|\zeta_n|/\zeta_n=1$ for $\zeta_n=0$.
Now $f(z) = 2/( B(z)+2 )$ is a bounded solution of \eqref{eq:de2} with
\begin{equation*}
A(z) = \frac{2 B''(z) + B''(z)B(z) - 2 \big( B'(z) \big)^2}{\big( B(z)+2 \big)^2}, \quad z\in\D.
\end{equation*}
Since $B$ is bounded, we have 
$A\in H^\infty_2$. The same construction was also used in the proof of 
\cite[Theorem~8]{GH:2012}.  Since $f'(z) = -2 B'(z)/( B(z)+2)^2$, we deduce:

(i) If $\{ \zeta_n \}_{n=1}^\infty$ is a Blaschke-sequence such that for each $N\in\N$ there exists a~point whose multiplicity is
greater than $N$, then $f'$ has zeros of arbitrarily high multiplicity.

(ii)
If $\{ \zeta_n \}_{n=1}^\infty$ is a Blaschke-sequence, which contains two subsequences of two-fold points whose pair-wise 
separation becomes arbitrarily small near the boundary $\partial\D$, then the distinct critical points of $f$ need not to
obey any pre-given separation. 
\phantom{a}\hfill $\diamond$
\end{example}

The classical result \cite[Theorem~8.2.2]{H:1997} due to C.-T. Taam,
whose proof is based on Sturm's comparison theorem,
implies a positive answer to the question (Q3). 
We take the opportunity to state a parallel result with an alternative proof.
Our method also produces an estimate for the behavior of solutions
near critical points.


\begin{theorem} \label{thm:main}
Let $A$ be analytic in $\D$, and let $\psi : [0,1) \to (0,1)$ be a non-increasing function
such that \eqref{eq:smoothness} holds. If $A$ satisfies \eqref{eq:grest}, 
then the hyperbolic distance between any zero $\zeta\in\D$ and any critical point $a\in\D$
of any non-trivial solution $f$ of \eqref{eq:de2} satisfies
\begin{equation} \label{eq:estimate}
\varrho_h(\zeta,a) \geq \frac{1}{2} \log\frac{1+\psi(|a|)/\max\{K\sqrt{M},1\}}{1-\psi(|a|)/\max\{K\sqrt{M},1\}}.
\end{equation}
\end{theorem}

Recall that
\begin{equation*}
S_g = \left( \frac{g''}{g'} \right)^\prime - \frac{1}{2} \left( \frac{g''}{g'} \right)^2
\end{equation*}
is the Schwarzian derivative of the meromorphic function $g$.


\begin{oldlemma}[\protect{\cite[p.~91]{L:1987}}] \label{lemma:lehto}
Let $g$ be meromorphic in $\D$ and satisfy $g''(0)=0$, $\nm{S_g}_{H^\infty_2} \leq 2$
and $(1-\abs{z}^2)^2 \abs{S_g(z)} \leq 1$ in some neighborhood $\abs{z}\leq \rho < 1$ of the origin.
Then $g$ is analytic in $\D$, and
\begin{equation*}
\abs{g'(z)} \leq \frac{S\abs{g'(0)}}{1-\abs{z}^2} \left( \log\frac{1+\abs{z}}{1-\abs{z}} \right)^{-2},
\quad z\in\D,
\end{equation*}
where $S=S(\rho)$ is a constant such that $0<S<\infty$.
\end{oldlemma}


\begin{proof}[Proof of Theorem~\ref{thm:main}]
Let $f$ be a solution of \eqref{eq:de2}, where
the coefficient $A$ satisfies~\eqref{eq:grest}. Let $f^\star$ be a solution of \eqref{eq:grest},
linearly independent to $f$, such that $W(f,f^\star)=1$. If we define $w=f^\star/f$, then
\begin{equation*}
S_w = 2A, \quad w' = \frac{1}{f^2} \quad \text{and} \quad
w'' = - 2 \, \frac{f'}{f^3}.
\end{equation*}
We conclude that $w''(z)=0$ if and only if $z\in\D$ is a critical point of $f$. Note that~$f$ 
does not vanish at the critical points.

If $f$ does not have any critical points in $\D$, then there is nothing to prove.
Let $a\in\D$ be a critical point of $f$, and consider two separate cases.

\subsubsection*{Case $a=0$}
Define the meromorphic function $g$ in $\D$ by $g(z) = w \big( \psi(0)rz \big)$,
where $r=1/\max\{ K \sqrt{2M}, 1\}$.
Since $w''(0)=0$, we conclude that  $g''(0)=0$. Now
\begin{align*}
  (1-\abs{z}^2)^2 |S_{g}(z)| & = 
  (1-\abs{z}^2)^2 \big| S_w \big( \psi(0) rz \big) \big| 
  \, \psi(0)^2 r^2\\
  & \leq 2M \bigg( \frac{\psi(0)}{\psi\big( \psi(0) \big)}\bigg)^2 r^2 \leq 1,
  \quad z\in\D.
\end{align*}
By Lemma~\ref{lemma:lehto} we conclude that $g$ is analytic in $\D$, which means that
$w$ does not have any poles in the pseudo-hyperbolic disc $\Delta_p(0,\psi(0)r)$, and
\begin{equation*}
\big| w'\big( \psi(0)rz \big)\big| \, \psi(0)r 
 \leq \frac{S |w'(0)| \psi(0) r}{1-\abs{z}^2}\, \left( \log\frac{1+\abs{z}}{1-\abs{z}} \right)^{-2},
\quad z\in\D.
\end{equation*}
Since $f^2 = 1/w'$, we deduce
\begin{align*}
\frac{\big| f\big( \psi(0)rz)\big) \big|^2}{|f(0)|^2} 
 & \geq \frac{1}{S}\,  (1-\abs{z}^2) \left( \log\frac{1+\abs{z}}{1-\abs{z}} \right)^{2},
\quad z\in\D,
\end{align*}
and hence $f$ has no zeros in $\Delta(0,\psi(0)r)$.

\subsubsection*{Case $a\neq 0$}
Define the meromorphic function 
\begin{equation*}
g_a(z) = \frac{1}{w \big(\varphi_a(\psi(|a|)rz) \big)-C_a}, \quad z\in\D,
\end{equation*}
where $r=1/\max\{ K \sqrt{2M}, 1\}$, and $C_a = w(a)- w'(a)(1-\abs{a}^2)/\overline{a}$ 
is a complex constant. Note that $C_a\neq w(a)$, since $w'(a)\neq 0$. Furthermore,
the choice of $C_a$ yields $g_a''(0)=0$. We obtain
\begin{align*}
  (1-\abs{z}^2)^2 |S_{g_a}(z)| & = 
  (1-\abs{z}^2)^2 \big| S_w \big( \phi_a( \psi(|a|) rz) \big) \big| 
  \, \big| \phi_a'\big( \psi(|a|) rz \big) \big|^2 \psi(|a|)^2 r^2\\
  & \leq 2M \left( \frac{\psi(|a|)}{\psi\Big( \frac{|a|+\psi(|a|)}{1+|a| \psi(|a|)}\Big)}\right)^2 r^2 \leq 1,
  \quad z\in\D.
\end{align*}
By Lemma~\ref{lemma:lehto} we conclude that $g_a$ is analytic in $\D$, which means that
$w$ does not attain the value $C_a$ in the pseudo-hyperbolic disc $\Delta_p(a,\psi(|a|)r)$, and further
\begin{align*}
& \frac{\big| w'\big(\phi_a(\psi(|a|)rz)\big) \big| \big| \phi_a'(\psi(|a|)rz)\big| \psi(|a|)r}{\big| w\big( \phi_a(\psi(|a|)rz) \big)-C_a\big| ^2}\\
& \qquad \leq \frac{S_a |w'(a)| (1-|a|^2) \psi(|a|) r}{|w(a)-C_a|^2} \, \frac{1}{1-\abs{z}^2} \left( \log\frac{1+\abs{z}}{1-\abs{z}} \right)^{-2},
\quad z\in\D.
\end{align*}
Since $f^2 = 1/w'$, we deduce
\begin{align*}
\frac{\big| f\big(\phi_a(\psi(|a|)rz)\big) \big|^2}{|f(a)|^2} 
 & \geq \frac{1}{S_a} \, \frac{|w(a)-C_a|^2}{\big| w\big( \phi_a(\psi(|a|)rz) \big)-C_a\big| ^2}
 \, \frac{\big| \phi_a'(\psi(|a|)rz)\big|}{1-|a|^2 } \\ 
& \qquad \times  (1-\abs{z}^2) \left( \log\frac{1+\abs{z}}{1-\abs{z}} \right)^{2}, \quad z\in\D,
\end{align*}
and hence $f$ has no zeros in $\Delta_p(a,\psi(|a|)r)$. The claim follows.
\end{proof}

The assertion converse to Theorem~\ref{thm:main} is false. If $f$ is an analytic non-vanishing function, 
then zeros and critical points of $f$ are trivially separated from each other. However, regardless
of the existence of the zero-free solution, the coefficient $A$ can grow arbitrarily fast.
This follows easily by considering compositions of the exponential function, for example.


\begin{remark}
An estimate for the separation between zeros and critical points, which turns out to be weaker than \eqref{eq:estimate}, 
is immediately available by using the theory of $\varphi$-normal functions, since
\begin{equation*}
\left( \frac{f'}{f} \right)^{\#} = \frac{\big| (f'/f)'(z) \big|}{1+\big| (f'/f)(z) \big|^2}
       \leq \frac{\big| f''(z) f(z) \big| + \big|f'(z)\big|^2}{\big| f(z) \big|^2 + \big|f'(z)\big|^2 }
       \leq \big| A(z) \big| + 1, \quad z\in\D.
\end{equation*}
See \cite[Theorem~4]{AR:2011}. Normal functions are considered further in Section~\ref{sec:normal}.
\end{remark}

Note the following  special case of Theorem~\ref{thm:main} (or \cite[Theorem~8.2.2]{H:1997}).


\begin{corollary}
If $A\in H^\infty_2$, then 
the hyperbolic distance between any zero and any critical point of any non-trivial solution
of \eqref{eq:de2} is uniformly bounded away from zero.
\end{corollary}

The following examples examine the sharpness of \eqref{eq:estimate}.


\begin{example}
Let $0<\gamma<\infty$. Then, the differential equation \eqref{eq:de2} with 
$$A(z)=(1+4\gamma^2)/(1-z^2)^2, \quad z\in\D,$$
admits the solution
\begin{equation*}
f(z)= \sqrt{1-z^2} \, \sin\left( \gamma \log\frac{1+z}{1-z} \right), \quad z\in\D,
\end{equation*}
whose zeros $\zeta_n = (e^{\pi n/\gamma}-1)/(e^{\pi n/\gamma}+1)$ are real for all $n\in\Z$
\cite[p.~162]{S:1955}. The hyperbolic distance
between two consecutive zeros is precisely $\pi/(2\gamma)$. Since~$f$ is a~real differentiable function 
on the real axis, we conclude that $f$ has a critical point in each open interval $(\zeta_n,\zeta_{n+1})$, $n\in\N$.
This means that there exists a sequence $\{ a_n\}_{n=1}^\infty$ of real critical points of $f$ such that
$\varrho_h(\zeta_n,a_n)$ remains uniformly bounded above  as $n\to\infty$. \hfill $\diamond$
\end{example}


\begin{example}
Let $1<q<\infty$. Then, the differential equation \eqref{eq:de2} with 
$$A(z)=\big( p'(z) \big)^2 + \frac{1}{2} \, S_p(z), \quad p(z) = \left( \log\frac{e}{1-z} \right)^q, \quad z\in\D,$$
admits the solution
\begin{equation*}
f(z)=\frac{1}{\sqrt{p'(z)}} \, \sin\big( p(z) \big), \quad z\in\D,
\end{equation*}
whose zeros $\zeta_n = 1-\exp( 1- (n\pi)^{1/q} )$, $n\in\Z$, are real
\cite[Example~12]{CGHR:2013}. 
Since~$f$ is a~real differentiable function 
on the real axis, there exists a sequence $\{ a_n\}_{n=1}^\infty$ of real critical points of $f$ such that
$a_n\in (\zeta_n,\zeta_{n+1})$ for $n\in\N$, and
\begin{equation*}
\varrho_h(\zeta_{n+1},a_n)  \leq \varrho_h(\zeta_{n+1}, \zeta_{n}) \sim \frac{\pi}{2q}\,  \big(n\pi\big)^{1/q-1}, \quad n\to\infty,
\end{equation*}
while
\begin{equation*}
\begin{split}
 \frac{1}{2} \log\frac{1+\psi(|a_n|)/\max\{K\sqrt{M},1\}}{1-\psi(|a_n|)/\max\{K\sqrt{M},1\}}
& \sim \frac{\psi(|a_n|)}{\max\{K\sqrt{M},1\}} 
\geq \frac{\psi(z_{n+1})}{\max\{K\sqrt{M},1\}}\\
& \sim \frac{ (n\pi)^{1/q-1}}{\max\{K\sqrt{M},1\}}, \quad n\to\infty.
\end{split}
\end{equation*}
Here $\psi(r) = 2^{-1} \big(\log(e/(1-r))\big)^{1-q}$ satisfies \eqref{eq:smoothness} for $K=\big(\log(2e)\big)^{q-1}$,
and $M=M(q)$ is the constant in \eqref{eq:grest}.
We conclude that, in this case, both sides of \eqref{eq:estimate} are asymptotically of the same 
order of magnitude.
\hfill $\diamond$
\end{example}


\section{Normality of solutions} \label{sec:normal}

A meromorphic function $f$ in $\D$ is normal (in the sense of Lehto and Virtanen)~if 
\begin{equation*} 
  \sup_{z\in\D} \, (1-|z|^2) \, f^{\#}(z)<\infty,
\end{equation*}
where $f^{\#} = |f'|/(1+|f|^2)$ is the spherical derivative of $f$. 
For more details on normal functions, we refer to \cite{C:1989,LV:1957}.

Assume that $A \in H^\infty_2$.
Let $f_1$ be a~non-trivial solution of \eqref{eq:de2} whose zero-sequence is $\{\zeta_n\}_{n=1}^\infty\subset \D$.
By \cite[Proposition~7]{GNR:preprint}, $f_1$ is normal if and only if
\begin{equation} \label{eq:preprintnormal}
\sup_{n\in\N} \, (1-|\zeta_n|^2) |f_1'(\zeta_n)| < \infty.
\end{equation}
Equivalently,
$f_1$ is normal if and only if
\begin{equation*}
\sup_{n\in\N}\, (1-|\zeta_n|^2) \, \frac{1}{|f_2(\zeta_n)|} <\infty,
\end{equation*}
where $f_2$ is any solution of \eqref{eq:de2} which is linearly independent to $f_1$.

By solving a certain interpolation problem, we conclude our main result.
Recall that the sequence $\{\zeta_n\}_{n=1}^\infty\subset \D$
is called uniformly separated if
\begin{equation*}
  \inf_{k\in\N} \, \prod_{n\in\N\setminus\{k\}} \varrho_p(\zeta_n,\zeta_k)>0,
\end{equation*}
while the Hardy space $H^p$ for $0<p<\infty$ consists of those analytic functions $f$ in~$\D$
for which
\begin{equation*}
\nm{f}_{H^p} = \lim_{r\to 1^-} \left( \frac{1}{2\pi} \int_0^{2\pi} |f(re^{i\theta})|^p \, d\theta \right)^{1/p}< \infty.
\end{equation*}


\begin{theorem} \label{thm:mmain}
Let $\{\zeta_n\}_{n=1}^\infty\subset\D$ be a uniformly separated sequence having infinitely many points.
Then, there exists $A\in H^\infty_2$ such that
\eqref{eq:de2} admits a solution~$f$ having the following properties:
\begin{enumerate}
\item[\rm (i)]
the zero-sequence of $f$ is $\{\zeta_n\}_{n=1}^\infty$;
\item[\rm (ii)]
$f$ belongs to the Hardy space $H^p$ for any sufficiently small $0<p<\infty$;
\item[\rm (iii)]
$f$ is non-normal.
\end{enumerate}
\end{theorem}

Note that normal solutions of \eqref{eq:de2} are known to possess some nice properties. For example, 
all normal solutions of \eqref{eq:de2} belong to the Hardy space $H^p$ for any sufficiently
small $0<p<\infty$ provided that $|A(z)|^2(1-|z|^2)^3 \, dm(z)$ is a Carleson measure; see \cite[Corollary~9]{GNR:preprint}
for the result and discussion on the Carleson measure condition.

The proof of Theorem~\ref{thm:mmain} relies on the following auxiliary result, which concerns interpolation. 
Recall that the space $\BMOA$ contains those functions $g\in H^2$ for which
\begin{equation*}
\sup_{a\in\D} \, \bnm{ g(\varphi_a(z))-g(a)}_{H^2} <\infty.
\end{equation*}


\begin{lemma} \label{lemma:interpolation}
Let $\{\zeta_n\}_{n=1}^\infty\subset\D$ be a uniformly separated sequence having infinitely many points, 
and assume that $\{w_n\}_{n=1}^\infty \subset \C$ satisfies
\begin{equation*}
\sup_{n\in\N} \, (1-|\zeta_n|^2)\, |w_n| \leq S < \infty.
\end{equation*}
Then, there exists $g = g(\{\zeta_n\}_{n=1}^\infty, \{w_n\}_{n=1}^\infty)\in\BMOA$ having the properties:
\begin{enumerate}
\item[\rm (i)]
$g'(\zeta_n) = w_n$ for $n\in\N$;\\[-0.3cm]
\item[\rm (ii)]
$\limsup\limits_{n\to\infty} \, {\rm Re}\, g(\zeta_n) = \infty$.
\end{enumerate}
\end{lemma}


\begin{proof}
Let $\xi\in\partial\D$ be an accumulation point of $\{\zeta_n\}_{n=1}^\infty$, and let
$B$ be the Blaschke product in \eqref{eq:Blaschke}. Since $\{\zeta_n\}_{n=1}^\infty$ is
uniformly separated, there exists a constant $0<\delta<1$ such that
\begin{equation} \label{eq:unisepp}
(1-|\zeta_n|^2) |B'(\zeta_n)| > \delta, \quad n\in\N.
\end{equation}
Define the sequence $\{\nu_n\}_{n=1}^\infty$ by
\begin{equation*}
\nu_n = \frac{1}{B'(\zeta_n)} \left( w_n - \frac{1}{\xi-\zeta_n} \right), \quad n\in\N,
\end{equation*}
and note that
\begin{equation*}
|\nu_n| \leq \frac{1-|\zeta_n|^2}{(1-|\zeta_n|^2)|B'(\zeta_n)|} \left( |w_n| + \frac{1}{|\xi-\zeta_n|} \right)
\leq \frac{S+2}{\delta} < \infty, \quad n\in\N.
\end{equation*}
Let $h$ be a bounded analytic function in $\D$ which solves the interpolation problem $h(\zeta_n)=\nu_n$ for $n\in\N$.
Existence of such function $h$ follows from \cite[Theorem~3]{C:1962}.
Finally, define
\begin{equation*}
g(z) = B(z) \, h(z) + \log \frac{1}{\xi-z}, \quad z\in\D.
\end{equation*}
Now $g\in\BMOA$, and $g'(\zeta_n) = B'(\zeta_n) \, \nu_n + 1/(\xi - \zeta_n) = w_n$ for
$n\in\N$. Moreover,
\begin{equation*}
\limsup_{n\to\infty} \, {\rm Re}\, g(\zeta_n) = \limsup_{n\to\infty} \, \log\frac{1}{|\xi-\zeta_n|} = \infty,
\end{equation*}
since $\xi$ is an accumulation point. 
\end{proof}


\begin{proof}[Proof of Theorem~\ref{thm:mmain}]
We apply a method which was also used in \cite[pp.~359--360]{GH:2012}.
Let $B$ be the Blaschke product in \eqref{eq:Blaschke}, and define the sequence $\{w_n\}_{n=1}^\infty$ by
\begin{equation*}
w_n = - \frac{1}{2} \frac{B''(\zeta_n)}{B'(\zeta_n)}, \quad n\in\N.
\end{equation*}
Now
\begin{equation*}
\sup_{n\in\N} \, (1-|\zeta_n|^2) |w_n| 
= \sup_{n\in\N} \, \frac{(1-|\zeta_n|^2)^2 |B''(\zeta_n)|}{2(1-|\zeta_n|^2)|B'(\zeta_n)|}
\leq \frac{\nm{B''}_{H^\infty_2}}{2\delta} < \infty,
\end{equation*}
where $\delta$ is the constant in \eqref{eq:unisepp}. 
Let $g$ be the function given by Lemma~\ref{lemma:interpolation}.

Define $f=Be^g$. By construction, $f$ has the prescribed zeros $\{\zeta_n\}_{n=1}^\infty$.
Since $g\in\BMOA$, $f\in H^p$ for any sufficiently small $0<p<\infty$ \cite[Theorem~1]{CS:1976}.
Following the argument in \cite[pp.~359--360]{GH:2012}, we conclude that
\begin{equation*}
A = - \frac{f''}{f} = - \frac{B''+2B'g'}{B} - (g')^2-g''
\end{equation*}
is analytic, since the interpolation property Lemma~\ref{lemma:interpolation}(i) guarantees that
$A$ has a~removable singularity at each point $\zeta_n$ for $n\in\N$. We also have $A\in H^\infty_2$,
since $\{\zeta_n\}_{n=1}^\infty$ is uniformly separated and $g\in\BMOA$;
see \cite{GH:2012} for more details.
Since $f$ is a solution of \eqref{eq:de2}
with $A\in H^\infty_2$, and
\begin{align*}
\limsup_{n\to\infty} \, (1-|\zeta_n|^2) |f'(\zeta_n)| 
         & = \limsup_{n\to\infty} \, (1-|\zeta_n|^2) |B'(\zeta_n)| \, e^{{\rm Re}\, g(\zeta_n)}\\
          & \geq \delta \cdot \limsup_{n\to\infty} \, e^{{\rm Re}\, g(\zeta_n)} = \infty,
\end{align*}
we conclude that $f$ is non-normal \cite[Proposition~7]{GNR:preprint}.
\end{proof}

The following result shows that two non-zero distinct values can be prescribed
for a normal solution of \eqref{eq:de2} under the restriction $A\in H^\infty_2$.
This result should be compared to \cite[Theorem~8]{GH:2012} in which one non-zero value is prescribed.


\begin{theorem} \label{thm:2pre}
Assume that $a,b\in\C$ are non-zero and distinct.
Let $\{\alpha_n\}_{n=1}^\infty\subset \D$ and $\{\beta_n\}_{n=1}^\infty\subset \D$ be two Blaschke
sequences, and let $B_{\alpha}$ and $B_{\beta}$ be the corresponding Blaschke products. 
If there exists a constant $0<\mu<1$ such that
\begin{equation} \label{eq:carleson}
|B_{\alpha} (z)| + |B_{\beta} (z)| \geq \mu >0, \quad z\in\D,
\end{equation}
then there is $A=A(a, b, \{\alpha_n\}_{n=1}^\infty, \{\beta_n\}_{n=1}^\infty)$
such that
$|A(z)|^2(1-|z|^2)^3 dm(z)$ is a~Carleson measure (thus $A\in H^\infty_2$)
and \eqref{eq:de2} admits a~normal solution $f$ for which
\begin{equation} \label{eq:fpre}
f(\alpha_n)=a, \quad f(\beta_n)=b, \quad n\in\N.
\end{equation}
\end{theorem}

Note that \eqref{eq:carleson} is satisfied, for example, if 
$\{\alpha_n\}_{n=1}^\infty \cup \{\beta_n\}_{n=1}^\infty$ is uniformly separated.
Of course, this is not necessary for \eqref{eq:carleson} to hold.


\begin{proof}[Proof of Theorem~\ref{thm:2pre}]
By \eqref{eq:carleson} and \cite[Theorem~2]{C:1962}, there exists a bounded analytic function
$h$ such that
\begin{equation} \label{eq:hpre}
h(\alpha_n) = 0, \quad h(\beta_n)=1, \quad n\in\N.
\end{equation}
Now, 
\begin{equation*}
f(z) = \exp\bigg( \log a + h(z) \, \log\frac{b}{a} \, \bigg), \quad z\in\D,
\end{equation*}
is a bounded and non-vanishing solution of \eqref{eq:de2}, where
\begin{equation*}
A = -\frac{f''}{f} = - \bigg( h' \cdot \log\frac{b}{a} \bigg)^2 - h'' \cdot  \log\frac{b}{a}.
\end{equation*}
Solution $f$ is normal by \eqref{eq:preprintnormal}, and \eqref{eq:fpre} follows from \eqref{eq:hpre}.
Since $h$ is bounded, we conclude that $\log f \in \BMOA$ and hence $|A(z)|^2(1-|z|^2)^3\, dm(z)$ is
a Carleson measure by \cite[Theorem~4(i)]{GNR:preprint}. In particular, $A\in H^\infty_2$.
\end{proof}

If $f$ is a normal meromorphic function, then by a result \cite{Y:1975} due to Yamashita
there exists a constant $C=C(f)$ such that
\begin{equation} \label{eq:balance}
\sup_{z\in\D}\, \left( (1-|z|^2) \, (f')^{\#}(z) \right) \left( (1-|z|^2) \, f^{\#}(z) \right) \leq C <\infty,
\end{equation}
while the converse statement is known to be false; see also \cite{L:1977}. 
By combining Theorem~\ref{thm:mmain}
and the following result, we conclude that a differential equation
\eqref{eq:de2} with $A\in H^\infty_2$ may admit a non-normal solution
satisfying \eqref{eq:balance}.


\begin{theorem} \label{thm:balance}
Suppose that $f$ is a solution of \eqref{eq:de2}, where $A$ is analytic in $\D$. Then
$(f')^\#(z) \, f^\#(z) \leq 4^{-1} \, |A(z)|$ for all $z\in\D$.
\end{theorem}


\begin{proof}
The claim is trivially true for the critical points of $f$. 
Similarly, at the zeros of~$f$ the claim follows from \eqref{eq:de2}.
For any $z\in\D$, for which $f'(z)\neq 0$ and $f(z)\neq 0$, 
\begin{align*}
  \abs{A(z)} & = \left| \frac{f''(z)}{f(z)} \right| 
  = \left| \frac{f''(z)}{f'(z)} \right|  \left| \frac{f'(z)}{f(z)} \right| 
  \geq 4 \, \frac{\left| \frac{f''(z)}{f'(z)} \right|}{\abs{f'(z)}^{-1} + \abs{f'(z)}}  
  \, \frac{\left| \frac{f'(z)}{f(z)} \right|}{\abs{f(z)}^{-1} + \abs{f(z)}}\\
  & = 4 \, (f')^\#(z) \, f^\#(z).
\end{align*}
Here we applied the inequality $x^{-1}+x\geq 2$ for $0<x<\infty$.
\end{proof}


\section{Normality of the quotient} \label{sec:quo}

Let $f_1$ and $f_2$ be linearly independent solutions of \eqref{eq:de2} with $A\in H^\infty_2$.
We may apply \cite[Theorem~3]{S:1955} to the zero-sequences of $f_1$ and $f_2$, 
but it is unclear whether these zero-sequences are separated from each other.
Concerning the case of the complex plane, see \cite[Theorem~2.6]{E:1999}.

The hyperbolic distance between any two distinct zeros of
$f_1f_2$ is known to be uniformly bounded away from zero, for example, 
if there exists a~constant $0<C<\infty$ such that
\begin{equation*}
(1-|z|^2)^2 |A(z)| \leq 1 + C(1-|z|), \quad z\in\D.
\end{equation*}
See \cite[Corollary~4]{GR:preprint}, which is essentially a restatement of \cite[Corollary, p.~328]{S:1983}.

If $f_1$ and $f_2$ are linearly independent solutions of \eqref{eq:de2}, then $w=f_1/f_2$ is
a~locally univalent meromorphic function in $\D$ such that the Schwarzian derivative $S_w$
satisfies $S_w=2A$; see \cite[Theorem~6.1]{L:1993}.
Hence, to prove that the hyperbolic distance between any two distinct zeros of $f_1f_2$ is uniformly
bounded away from zero, it is sufficient to show that the
hyperbolic distance between any
zero and any pole of $w$ is uniformly bounded away from zero. Recall that all normal functions satisfy this property
by the Lipschitz-continuity (as mappings from $\D$ equipped with the hyperbolic metric
to the Riemann sphere with the chordal metric).

Suppose that $w$ is a meromorphic function
satisfying $S_w \in H^\infty_2$. Does it follow that $w$ is normal?  In favor of the affirmative answer,
recall that $\nm{S_w}_{H^\infty_2} \leq 2$ implies that $w$ is univalent \cite[Theorem I]{N:1949}, and hence normal~\cite[p.~53]{LV:1957}.
As surprising as it is, the answer to this question is negative.
By a result \cite{L:1973} due to P.~Lappan, 
there exists a~uniformly locally univalent analytic function in $\D$, which is not normal.
In a subsequent paper \cite[Theorem~5]{L:1978}, Lappan presents a concrete function
having these properties. This function 
\begin{equation} \label{eq:intuni}
\left( 1-z \right)^{-\frac{1+10i}{100}} - \left( 1-z \right)^{-\frac{i}{100}}
\end{equation}
emerges as a primitive of a univalent function in $\D$.

If $w$ is meromorphic in $\D$, and there exists $0<\delta\leq 1$ such that $w$ is univalent
in each pseudo-hyperbolic disc $\Delta_p(a,\delta)=\set{z\in\D : \varrho_p(z,a)<\delta}$ for $a\in\D$,
then $w$ is called uniformly locally univalent. We give a short proof for the
following well-known lemma for the convenience of the reader.


\begin{oldlemma} \label{oldlemma:myl}
A meromorphic function $w$ in $\D$ satisfies $S_w\in H^\infty_2$
if and only if  $w$  is uniformly locally univalent.
\end{oldlemma}


\begin{proof}
Suppose that $w$ is meromorphic and $S_w\in H^\infty_2$.
If $\nm{S_w}_{H^\infty_2} \leq 2$, then the assertion follows from \cite[Theorem~I]{N:1949};
for the meromorphic case, see \cite[Corollary~6.4]{P:1975}.
If $\nm{S_w}_{H^\infty_2} > 2$, then define
$g_a(z)=w(\varphi_a(\delta z))$ for $a\in\D$ and $\delta= (2/\nm{S_w}_{H^\infty_2})^{1/2}$.
Now
\begin{equation*}
|S_{g_a}(z)| (1-|z|^2)^2 = \big| S_w\big( \varphi_a(\delta z) \big) \big| \big| \varphi_a'(\delta z) \big|^2 \delta^2 (1-|z|^2)^2
\leq \nm{S_w}_{H^\infty_2} \, \delta^2 = 2, \quad z\in\D,
\end{equation*}
and hence Nehari's theorem 
implies that $w$ is univalent in $\Delta_p(a,\delta)$ for any $a\in\D$.

Conversely, suppose that $w$ is meromorphic and uniformly locally univalent. Then, $A=2^{-1} S_w $ is analytic in $\D$, and
the hyperbolic distance between any two distinct 
zeros of any non-trivial solution of \eqref{eq:de2} is uniformly bounded away from zero. Now
\cite[Theorem~4]{S:1955} implies $S_w\in H^\infty_2$.
\end{proof} 

Since $S_w\in H^\infty_2$ does not imply that $w$ is normal, it is natural to ask
whether we can estimate the growth of the spherical derivative of $w$.
For example, if $w$ is the function in \eqref{eq:intuni}, then
\begin{equation*}
  \sup_{z\in\D} \, (1-\abs{z}^2)^2 \, w^\#(z) \leq \sup_{z\in\D} \, (1-\abs{z}^2)^2 \, \abs{w'(z)} < \infty
\end{equation*}
by the distortion theorem of analytic univalent functions \cite[p.~21]{P:1975}.
It turns out that the uniform local 
univalence does restrict the growth of the spherical derivative.
Due to an application of Cauchy-Schwarz inequality, we have a reason to believe that 
the estimate in Theorem~\ref{thm:spg} is not sharp.


\begin{theorem} \label{thm:spg}
Let $w$ be a meromorphic function in $\D$ such that $S_w\in H^\infty_2$. Then
\begin{equation} \label{eq:anormal}
\sup_{z\in\D} \, (1-\abs{z}^2)^\alpha \, w^\#(z) < \infty
\end{equation}
for all $\alpha$ for which $\left( 1+\nm{S_w}_{H^\infty_2}/2 \right)^{1/2} + 1 < \alpha < \infty$.
\end{theorem}


\begin{proof}
By assumption, the differential equation \eqref{eq:de2} with $A=2^{-1} S_w$ admits two linearly
independent solutions $f_1$ and $f_2$ such that $w=f_1/f_2$. By \cite[Theorem~2]{GR:preprint}, all
solutions $f$ of \eqref{eq:de2} satisfy
\begin{equation*} 
\sup_{z\in\D} \, (1-\abs{z}^2)^\alpha \, \abs{f(z)} < \infty, \quad \frac{\left( 1+\nm{S_w}_{H^\infty_2} /2\right)^{1/2} - 1}{2} < \alpha < \infty.
\end{equation*}
By the Cauchy integral formula,
\begin{equation} \label{eq:Cint}
\sup_{z\in\D} \, (1-\abs{z}^2)^\alpha \, \abs{f'(z)}^2 < \infty, \quad \left( 1+\nm{S_w}_{H^\infty_2} /2\right)^{1/2} + 1 < \alpha < \infty.
\end{equation}

As in \cite{S:2012}, we write
\begin{equation*}
w^\# = \frac{|w'|}{1+|w|^2} = \frac{1}{\frac{1}{|w'|}+\frac{|w^2|}{|w'|}} = \frac{|W(f_1,f_2)|}{|f_1|^2+|f_2|^2}.
\end{equation*}
By means of the Cauchy-Schwarz inequality, we deduce
\begin{equation*}
\begin{split}
|W(f_1,f_2)|^2 & = \big| f_1(z) f_2'(z) - f_1'(z) f_2(z) \big|^2
\leq \big( |f_1(z)| |f_2'(z)| + |f_1'(z)| |f_2(z)| \big)^2\\
& \leq \left( |f_1(z)|^2 + |f_2(z)|^2 \right) \left( |f_1'(z)|^2 + |f_2'(z)|^2 \right), \quad z\in\D.
\end{split}
\end{equation*}
In conclusion, $w^\# \leq \abs{W(f_1,f_2)}^{-1} \left( \abs{f_1'}^2 + \abs{f_2'}^2 \right)$. Now, the
assertion follows from~\eqref{eq:Cint}, when applied to $f_1$ and $f_2$.
\end{proof}

We deduce information related to the problem which is mentioned in the beginning of Section~\ref{sec:quo}. In particular,
if $f_1$ and $f_2$ are linearly independent solutions of \eqref{eq:de2} with $A\in H^\infty_2$,
then $w=f_1/f_2$ satisfies \eqref{eq:anormal} for some sufficiently large $\alpha=\alpha(\nm{A}_{H^\infty_2})$
with $1<\alpha<\infty$, by Theorem~\ref{thm:spg}. Now \cite[Theorem~4]{AR:2011} implies that
there exists a~constant $0<\delta<1$ such that
\begin{equation*}
\varrho_p(\zeta_1,\zeta_2) \geq \delta \cdot \max\Big\{ (1-|\zeta_1|^2)^{\alpha-1}, (1-|\zeta_2|^2)^{\alpha-1} \Big\}
\end{equation*}
whenever $\zeta_1,\zeta_2\in\D$ are points for which $f_1(\zeta_1)=0=f_2(\zeta_2)$.


\section{Acknowledgements}

The author thanks J.~Heittokangas for helpful conversations, and for pointing out the reference 
\cite[Theorem~8.2.2]{H:1997}.


\end{document}